\documentclass[12pt]{amsart}

\usepackage{amscd}
\usepackage{amsmath,amsfonts,amsthm,epsfig,latexsym,graphicx,amssymb,psfrag}
\usepackage[english]{babel}
\newtheorem{coro}{Corollary}
\newtheorem{defi}{Definition}[section]
\newtheorem{prop}{Proposition}[section]
\newtheorem{theo}{Theorem}
\newtheorem{lemm}{Lemma}[section]

\newtheorem{rema}{Remark}

\newtheorem{clai}{Claim}

\newfont{\df}{cmssbx10}

\usepackage{xcolor}

\def\R{I\kern -0.37 em R}
\def\N{I\kern -0.37 em N}
\def\Z{I\kern -0.37 em Z}

\def\supess_#1{\mathop{\rm supess}\limits_{#1}}
\def\infess_#1{\mathop{\rm infess}\limits_{#1}}

\def\AA{{\mathbb A}}  
\def\DD{{\mathbb D}}

 \def\NN{{\mathbb N}} 

 \def\RR{{\mathbb R}} \def\SS{{\mathbb S}}
\def\TT{{\mathbb T}}

 \def\ZZ{{\mathbb Z}}

\begin{document}

\title[Burnside problem on surfaces.]{\bf  Burnside problem for groups of homeomorphisms of compact surfaces. }
\author{Nancy Guelman and Isabelle Liousse} \thanks{This paper was partially
supported by the Labex CEMPI  (ANR-11-LABX-0007-01),  Universit\'{e} de Lille 1, PEDECIBA, Universidad de la
Rep\'{u}blica and IFUM}

\address{{\bf  Nancy Guelman}
IMERL, Facultad de Ingenier\'{\i}a, Universidad de la Rep\'ublica,
C.C. 30, Montevideo, Uruguay.  \emph{nguelman@fing.edu.uy}.}

\address{{\bf    Isabelle Liousse}, UMR CNRS 8524, Universit\'{e} de Lille1,
59655 Villeneuve d'Ascq C\'{e}dex,   France.  \emph {liousse@math.univ-lille1.fr}. }

\begin{abstract} A group $\Gamma$ is said to be periodic if for any $g$ in $\Gamma$ there is a positive integer $n$  with $g^n=id$.
 We first prove that a finitely generated periodic group acting on the 2-sphere $\SS^2$ by $C^1$-diffeomorphisms    with a  finite orbit,  is finite and conjugate to a subgroup of $\mathrm{O}(3,\R)$  and we  use  it for proving that a finitely generated periodic group of spherical diffeomorphisms with even bounded orders is finite.

Finally, we show that a finitely generated periodic group of homeomorphisms of any orientable compact surface other than the 2-sphere or the 2-torus (which is the purpose of a previous paper of the authors) is finite. \end{abstract}

\maketitle
\date{}

\section {Introduction.}

\medskip

\begin{defi}
A group $\Gamma$ is said to be {\bf periodic} if  any $g$ in $\Gamma$ has finite order, that is, there exists a positive integer $n$  with $g^n=id$.

\end{defi}

One of the oldest problem in group theory was first posed by William
Burnside in 1902 (see \cite{B}): {\em ``Let $\Gamma$ be a finitely generated periodic group.
Is $\Gamma$ necessary a finite group?" }

It is obvious that an abelian finitely generated periodic group is
finite.

In 1911, Schur (see \cite{Sh}) proved that this is true for
subgroups of $\mathrm{GL}(k,\mathbb C)$, $k\in \mathbb N$.

\smallskip

But, in general,  according to Golod (see \cite{Go}) the answer is negative.  Later,  Adjan and Novikov  (see \cite{AN}), Ol'shanskii, Ivanov  and  Lysenok  (see \cite{Ol}, \cite{Iv} and \cite{Ly}) exhibited many examples of infinite,  finitely generated and periodic groups  with even bounded orders.

\smallskip

The problem raised by Burnside is still open for groups of homeomorphisms (or
diffeomorphisms) on closed manifolds. Very few examples are known.

\medskip

As a corollary of H\"older theorem (see, for example, section 2.2.4 of \cite{Na}), it holds that  a finitely generated periodic  group of
circle homeomorphisms is finite. Note that, even in this case,  finiteness of the generating set is crucial: the group
consisting of all rational circle rotations is periodic and infinite.

\smallskip

Rebelo and Silva (see \cite{RS}) proved that any finitely generated
periodic subgroup of $C^2$-symplectomorphisms of a  compact
4-dimensional symplectic manifold is  finite, provided that
the fundamental class in $H^4 (M, Z)$ is a product of classes in $
H^1 (M, Z)$.

\smallskip

The authors proved (see \cite{GL3}) that any finitely generated periodic  subgroup  of $\mathrm{Homeo }_\mu(\TT^2)$ is finite, where  $\mu$ is a probability
measure on $\TT^2$ and $ \mathrm{Homeo }_\mu (\TT^2)$ is the subgroup of
orientation-preserving homeomorphisms of $\TT^2$ preserving $\mu$.  Moreover, they showed that every finitely generated 2-group of toral homeomorphisms is finite.

 As a consequence of this result, we get the following

 \begin{coro}   Any finitely generated periodic  subgroup  of $\mathrm{Homeo }(\AA^2)$ is finite, where  $\AA^2$ is the closed annulus.\end{coro}\label{Ann}

The idea of the proof is to  form the double of $\AA^2$ by gluing two copies of $\AA^2$ along their boundaries. This will be explained in section 5.

\smallskip

In this paper, we study   related questions. We first consider  finitely
generated periodic groups of diffeomorphisms of the 2-sphere,  $\mathbb{S}^2$  and prove

\begin{theo}
\label{theo1} \

Let $\Gamma$ be a finitely generated periodic group of $C^1$-diffeomorphisms
of $ \mathbb{S}^2$. If $\Gamma$ has a finite orbit then it is finite.

Moreover if $\Gamma$ consists in orientation preserving  $C^1$-diffeomorphisms and acts with  a global fixed point then it is abelian.
\end{theo}

\smallskip

In order to relax the differentiability hypothesis, we have as a consequence of the proof of Proposition \ref{Cardgeq3},  the following
 \begin{coro} \label{coroCardgeq3}
  If $G$ is a finitely generated periodic group of orientation preserving homeomorphisms
of $ \mathbb{S}^2$ and   $G$ has a finite orbit of cardinality at least 3, then it is finite.
\end{coro}

\smallskip

In some cases, we are able to establish the existence of a finite orbit, in particular we have the following

\begin{coro} \label{coroeven}
If $G$ is a finitely generated periodic group of spherical diffeomorphisms with even bounded orders then $G$ is finite.
\end{coro}

\medskip

Let $\Gamma$ be a finite group,  by a classical result of Kerekjarto and Eilenberg  (see \cite{CK},\cite{Ei}, \cite{Ep}, \cite{K}), every
topological action of $\Gamma$  on the 2-sphere is conjugate to an orthogonal action (i.e by orthogonal maps of $\mathrm{O}(3,\RR)$).

Moreover, as it is explained in  \cite{MS}, if the finite group $\Gamma$ acts smoothly on a closed surface, $\Gamma$ leaves invariant some Riemannian metric of
constant curvature. Thus any action of $\Gamma$ on the 2-sphere is
conjugate in  $\mathrm{Diff} ^1(\SS^2)$  to an orthogonal  action. We will refer to this result as the ``folkloric" one.

Combining these results with Theorem \ref{theo1}  and Corollary  \ref{coroCardgeq3} we have the following rigidity result.
\begin{coro}
Let $\Gamma$ be a finitely generated periodic group.

1. If  $\Gamma$  acts by $C^1$-diffeomorphisms of $ \mathbb{S}^2$ with a finite orbit then it is conjugate in  $ {\text {\em Diff }} ^1(\SS^2)$  to a finite orthogonal group, that is a finite subgroup of $\mathrm{O}(3,\RR)$.

2. If $\Gamma$  acts by  orientation preserving homeomorphisms
of $ \mathbb{S}^2$ with a finite orbit of cardinality at least 3, then it is  conjugate   to a finite orthogonal group.
\end{coro}

\medskip

We note that the only "interesting compact surfaces" for studying Burnside Problem are the sphere and the torus. In section 5 we will prove
 \begin{theo}
\label{theo2}
Any finitely generated periodic group acting by homeomorphisms on a compact orientable surface $S$ of genus $g \geq 2$, is a finite group.

\end{theo}
As it can be verified in Section 5,  this kind of results for orientable compact surfaces are consequences of well known results.
One of them (see \cite{Ep}) is that any periodic homeomorphism in any compact surface preserves a Riemannian metric of constant curvature. By Killing-Hopf Theorem (see, for example section 6.2 of \cite{St}) the curvature is -1 if the genus of the surface is greater than one and it is 0 for the torus and it is 1 for the sphere.

As Katherine Mann showed us, the proof of Theorem \ref{theo1} can be extended to any manifold, so we have the following
\begin{theo}
\label{theo3} \

 Let $M$ be a  compact riemannian  manifold and let $\Gamma$ be a finitely generated periodic group of $C^1$-diffeomorphisms of $M$. If $\Gamma$ acts on $M$ with a finite orbit or preserves a finite union of circles  then  $\Gamma$ is finite. \end{theo}

As consequence we get

\begin{coro} \label{coro21} \


Let $\Gamma$ be a finitely generated periodic group of orientation preserving $C^1$-diffeomorphisms of $\SS^2$. If $\Gamma$ has  a non trivial center $Z(\Gamma)$  then  $\Gamma$ is finite.
\end{coro}

\medskip

Theorem \ref{theo1} will be proved in sections 2, 3, 4 and section 5 is devoted to orientable compact surfaces, in particular this section contains the proofs of Corollary \ref{Ann} and Theorem \ref{theo2}. Theorem \ref{theo3} and Corollary \ref{coro21}  will be proved in section 6.  Corollary \ref{coroeven}  will be proved in section 7.

\medskip

\noindent{\bf Acknowledgements.} We are grateful to Andr\'es Navas for proposing to us this subject and for fruitful discussions.  We thank K. Parwani for suggesting us the Burnside problem on the closed annulus and the closed disk. We also thanks to F. Leroux, J. Franks and K. Mann for several useful suggestions.

 \section{Classification of finite order orientation preserving spherical homeomorphisms.}

\smallskip

In this section, we recall the classification of finite order orientation preserving spherical homeomorphisms of sphere
up to conjugacy.

According to \cite{CK}, \cite{Ei}, \cite{Ep} or \cite{K}, a finite order spherical homeomorphism  is conjugate to an orthogonal map of $\mathrm{O}(3,\RR)$ and using the classification of elements in the orthogonal group $\mathrm{O}(3,\RR)$, wet get the following  definition and  proposition.


\begin{defi} A spherical  homeomorphism is called  {\bf quasi-rotation}  if it is
conjugated to a spherical rotation. \end{defi}

\begin{prop}\label{isotop}

A finite order, orientation preserving spherical homeomorphism is a quasi-rotation
and then if it is non trivial it has exactly two fixed points.
 \end{prop}

\begin{rema}\label{relrk} The ``folkloric result"  states that if the homeomorphism is a $C^1$-diffeomorphism then the conjugating map is also a $C^1$-diffeomorphism.
\end{rema}


As a corollary of Proposition \ref{isotop}, we get

\begin{coro}\label{pseudo}

If $\Gamma$ is a periodic group, then $\mathbf {G}$ the subgroup $ \{g\in \Gamma : g$ is orientation preserving  $\}$ is exactly the set $\{g\in \Gamma  : g$ is a quasi-rotation  $\}$. In particular, the subset of $\Gamma$ consisting in quasi-rotations is a subgroup.
\end{coro}

\section{Reduction to  groups of rational quasi-rotations.}

\smallskip

A direct consequence of the fact that the composition of two orientation reversing homeomorphisms is  orientation preserving, is

\begin{lemm}\label{finiteindex}

Let $\Gamma$ be a finitely generated periodic group of  spherical  homeomorphisms. Then $G$, the subgroup of $\Gamma$  consisting  in quasi-rotations is of finite index in $\Gamma$. Hence, $\Gamma$ is finite if and only if  $G$ is finite. \end{lemm}

For proving Theorem \ref{theo1}, it is enough to prove that $G$ is finite. This is the purpose of the following section.
Our proof consists in considering the following three cases: the case where $G$ has a global fixed point achieved in Proposition \ref{Global},
the case where $G$ has a finite orbit of cardinality 2 proven in  Proposition \ref{card2} and the remaining case where $G$ has a
finite orbit of cardinality at least 3 showed in Proposition \ref{Cardgeq3}.

\section{Burnside problem for groups of rational quasi-rotations.}

\smallskip

Let $G$ be a finitely generated periodic group of quasi-rotations of the sphere.

\begin{defi}
We denote by $\mathbf {P_G}$ the set of points in $\mathbb S^2$ that arise as fixed points of some non trivial element of $G$.

Let $x\in P_G$, we denote by $\mathbf { St_G }(x) $ the {\bf stabilizer} in $G$ of $x$, that is the set:
 $$\mathbf { St_G }(x) =\{g\in G : g(x)=x \}.$$ \end{defi}

\begin{lemm}\label{properties}

The set $P_G$ is $G$-invariant and $\displaystyle G=\bigcup_{x\in P_G} \mathrm{St}_G (x)$.
\end{lemm}

\begin{proof}

Let $x_0\in P_G$ and $g\in G$. By definition, there exists $f\in G$ such that $f(x_0)= x_0$. As $g\circ f\circ g^{-1} (g(x_0)) =g\circ f(x_0) = g(x_0)$, it follows that $g(x_0) \in P_G$.

The second point is a direct consequence of the fact that any element in $G$ admits a fixed point (every non trivial quasi-rotation has exactly two fixed points). \end{proof}

\medskip

\subsection{Groups acting with a  global fixed point.} \

\smallskip

 \begin{defi} Let $x\in \SS^2$, we define the following groups:

 ${\text{\bf Diff}} ^{1} _+ (\SS^2)$ consisting in  orientation preserving $C^1$ spherical diffeomorphisms,

 ${\text{\bf Diff}}^1_{x}(\SS^2)$ consisting in  $C^1$ spherical diffeomorphisms fixing $x$ and

${\text {\bf Diff}}^{1}_{x,+} (\SS^2)$ their intersection subgroup.

    \end{defi}

  \medskip

This section contains results on finitely generated periodic subgroups of $\mathrm{ Diff}^1_{x,+} (\SS^2)$. Furthermore, in the next section, we will apply these results in order to describe stabilizers of points that might not be finitely generated.

 \medskip

 \begin{defi} Define the map $\mathbf {D}: \mathrm{Diff}^1_{x} (\SS^2) \rightarrow \mathrm{GL}(2,\R)$ by
  $$D(g) = Dg(x) : \RR^2 \approx T_x \SS^2\rightarrow \RR^2 \approx T_x \SS^2, \text{ the differential map at $x$}.$$

\end{defi}

\medskip

\begin{lemm}\label{Diff}
The map $D$ is a morphism and the image of a periodic subgroup of $\mathrm{Diff}^1_{x,+} (\SS^2)$ is a  periodic subgroup of $\mathrm{SL}(2,\R)$.
\end{lemm}

\begin{proof}

As $D(fg) (x) = Df(g(x)) Dg(x)= Df(x)Dg(x)$  for any $f$, $g$ in $\mathrm{Diff}^1_{x,+} (\SS^2)$,  $D$ is a morphism. Moreover, $D(g)$ has finite order provided that $g$ has.

\medskip

Let  $g$ be a finite order element in $\mathrm{Diff}^1_{x,+} (\SS^2)$. By Proposition \ref{isotop} and Remark \ref{relrk},   there exist a spherical diffeomorphism $h$ and a spherical rational rotation  $R_\alpha$ such that $g=h^{-1} R_\alpha h $. Without loss of generality, we can assume that $h(x)=x$ and $R_\alpha (x)=x$.

\smallskip

Indeed, if not  $y:=h(x)\not=x$. There exists a spherical  rotation  $R_\beta$ such that $R_\beta(y) =x$ and therefore $ R_\beta h (x)=x$ and we  can rewrite  $g =( R_\beta  h) ^{-1} (R_\beta R_\alpha R_\beta ^{-1}) (R_\beta h).$

Thus  $g= H^{-1} R_\alpha' H$, where $H=R_\beta h$ fixes $x$ and $R_\alpha'=  R_\beta R_\alpha R_\beta ^{-1}$  is a spherical rotation and   $R_\alpha' (x)  = H g H^{-1}(x)= H g(x) = H(x) =x$.

\medskip

Finally, we have  $D(g)= D(h) ^ {-1} D(R_\alpha)  D(h)$, since $D$ is a morphism. The linear map $D(R_\alpha)$ is the  planar rotation of angle $ \alpha$, then $D(g)$ has determinant equal to $1$ so it belongs to $ \mathrm{SL}(2,\R)$. \end{proof}

\smallskip


\begin{prop}\label{Global} Let $G$ be a finitely generated periodic subgroup of  $\mathrm{Diff}_{x,+}^{1} (\SS^2)$.  Then $G$ is finite and abelian.
 \end{prop}

\begin{proof}

The set $D(G)$ is a finitely generated periodic subgroup, since it is the image by a morphism of $G$ satisfying these properties.
According to Schur's theorem (\cite{Sh}), as $D(G)$ is a finitely generated periodic subgroup of $ \mathrm{SL}(2,\R)$, it is finite then it is compact.

The classification of compact subgroups of $\mathrm{SL}(2,\R)$ states that $D(G)$ is conjugated to a subgroup of $\mathrm{SO}(2,\R)$ consisting in linear planar rotations (see for example \cite{La}), hence $D(G)$ is abelian.

As a consequence, for any $f,g$ in $G$,  $D([f,g]) = Id$, where $[f,g]=fgf^{-1}g^{-1}$ is the commutator of $f$ and $g$.

\smallskip

Finally, $[f,g]=h R_\alpha h ^{-1}$, since it is a quasi-rotation.
Then  $D([f,g]) = D(h R_\alpha h ^{-1}) =  D(h ) D(R_\alpha) D(h )^{-1} = Id $ and therefore  $D(R_\alpha) =Id$ so  $\alpha=0$ and $R_\alpha=Id$. Hence  $[f,g] = Id$.

Consequently $G$ is a periodic, finitely generated and abelian group. It follows that $G$ finite.  \end{proof}

\subsection{$G$ has a finite orbit of cardinality  $2$} \

\smallskip
An important ingredient in this case is the following lemma concerning stabilizers of points. As a consequence of Proposition \ref{Global}, we have

\begin{lemm}\label{abelian}

Let  $x_0\in P_G$ and $G'$ be a  periodic subgroup of  $\mathrm{Diff}_{x_0,+}^{1} (\SS^2)$, then $G'$ is an abelian group and its elements  have the same two fixed points.  Moreover, any finitely generated subgroup of $G'$ is finite and conjugated to a group of rational spherical rotations of same axis.

In particular, if $G$ is a finitely generated periodic subgroup of  $\mathrm{Diff} _+^{1} (\SS^2)$, the subgroup $G':= \mathrm{St}_G (x_0)$ is an abelian group.
\end{lemm}

\begin{proof}

Let $f$, $g$ in $G'$ . The group $<f,g>$ generated by $f$ and $g$ is a finitely generated periodic subgroup of $\mathrm{Diff} _+^1 (\SS^2)$ that fixes $x_0$. Hence, according to Proposition \ref{Global}, $<f,g>$ is finite and abelian. Consequently, $f$ and $g$ commute.

Moreover, $f g f^{-1} = g$ implies that $Fix (g) = f( Fix g) $. Let $y\not=x$ be the second fixed point of the quasi-rotation $g$. Then $Fix(g) = \{x,y \} = \{ f(x)=x, f(y) \}$, then $f(y)= y$.

\medskip

Any finitely generated subgroup of  $G'$ is abelian and periodic, so it is a finite subgroup of $ \mathrm{Diff} _+^1 (\SS^2)$.
 By  the folkloric result, it is $C^1$-conjugated to a subgroup of rational rotations.

\end{proof}

\begin{prop}\label{card2}

If $G$ is a finitely generated periodic subgroup of  $C^1$  quasi-rotations. If $G$ has  a finite orbit of cardinality  $2$, then  $G $ is finite.
\end{prop}

\begin{proof}

Let $x_0$ be a point  with  $G$-orbit of cardinality 2.  We write $\mathcal O_G({x_0}) =\{x_0, x'_0\}$.

\smallskip

By Corollary \ref{abelian},  $\mathrm{St}_G ({x_0})$ is abelian.

\smallskip

We claim  that $[G,G]$, the first derivated subgroup of $G$, is contained in $\mathrm{St}_G ({x_0})$.

Let $g\in G$, we have  $g(x_0) \in \mathcal O_G({x_0}) =\{x_0, x'_0\}$ and
 $g(x'_0) \in \mathcal O_G({x_0}) =\{x_0, x'_0\}$. Noting that if $g(x_0) = x'_0$ then $g^{-1}(x_0) = x'_0$, it is easy to check that $[f,g] (x_0)= f^{-1} g^{-1} f g (x_0) = x_0$, in all possible cases.

 \medskip

We conclude that $[G,G]$ is abelian, this means that $[[G,G],[G,G]]$, the second derivated group of $G$, is trivial.

The last part of this proof is a general fact for finitely generated groups generated by $s$ finite order elements $g_1,..., g_s$:
{\em  ``any element of $G$ can be written $g= g_1^{p_1}....g_s^{p_s} C$, where $C\in [G,G]$ and $p_i\geq 0$ is bounded by the order of $g_i$."}
 So the index of $[G,G]$ in $G$ is bounded by the product of the orders of $g_1,..., g_s$. Moreover,  Schreier's lemma
states that any subgroup of finite index in a finitely generated
group is finitely generated. This implies that  $[G,G]$ is also finitely generated.

Here, as $G$ is a periodic group then $[G,G]$ is a finitely generated periodic group.
In particular, it is  generated by finite order elements. Hence, last argument shows that  $[[G,G],[G,G]]$ has finite index in $[G,G]$ which has finite index in $G$.  Finally $[[G,G],[G,G]]$ has finite index in $G$. But, we also have shown that $[[G,G],[G,G]]$ is trivial, so $G$ is finite.
\end{proof}

\medskip

\begin{rema} Under the hypotheses that $G$ is a finitely generated periodic group of homeomorphisms, $\#P_G= 2$ and $G$ preserves a probability measure on $\mathbb S^2\setminus P_G$, analogous arguments as those developed in \cite{GL3} show that $G$ is finite and abelian.
\end{rema}

\smallskip
A sketch of the proof of last Remark is the following:
$G$ acts on the open annulus $\mathbb S^2\setminus P_G$ and preserves a measure. Hence, we can define the  rotation map $\rho : G \rightarrow \mathbb S^1$; the number $\rho(g)$ coincides with the angle of a rotation conjugated to $g$.
As in  \cite{GL3}, one shows that $\rho$ is a morphism. Therefore, it vanishes on commutators. Then
any commutator is  conjugate to a rotation of angle $0$, so it is trivial. It follows that  $G$ is abelian  and since it is also finitely
 generated and periodic, it is finite.

\subsection{$G$ has a finite orbit of cardinality at least $3$}

\medskip

\begin{prop}\label{Cardgeq3}

Let $G$ be a finitely generated periodic subgroup of  quasi-rotations.  If $G$ admits a finite orbit of cardinality at least $3$, then  $G $ is finite.

\end{prop}

\begin{proof}

Let $x_0\in\SS^2$ having a finite $G$-orbit.

\smallskip

 As $\#\mathcal O_G({x_0})\geq 3$,  we can  write $\mathcal O_G({x_0}) =\{x_0=g_0(x_0), x_1=g_1(x_0),..., x_n= g_n(x_0)\} $, with distinct $x_i$ and $n\geq 2$.

\medskip

  We first claim that if $G$ is not finite, then $x_0 \in P_G$.

\smallskip

Indeed, as $G$ is not finite, there exists $f\notin \{Id=g_0, g_1, ..., g_n\} $. Since $f(x_0) \in \mathcal O_G({x_0})$, there exists $i$ such
that $f(x_0)= g_i(x_0)$, and then $x_0$ is fixed by $g_i ^{-1}f $.

\medskip

 We secondly prove that $\mathrm{St}_G (x_0)$ is a finite group.

\smallskip

If $\mathrm{St}_G (x_0)$ is not a finite group, we write  $\mathrm{St}_G(x_0) = \{f_n, n\in \mathbb N\}$, where $f_n\not=f_m$ if $n\not= m$.

As $\mathcal O_G({x_0})$ is finite, there are infinitely many  $f_{s_n} , n\in \mathbb N$ such that $f_{s_n}(x_1)=f_{s_m}(x_1)$, for any $n,m$.  Then $f_{s_0}^{-1} \circ f_{s_n} (x_1)=x_1$ and so $F_n= f_{s_0}^{-1} \circ f_{s_n}$ fixes $x_0$ and $x_1$.

\smallskip

 Analogously, there exist infinitely many  $F_{k_n} , n\in \mathbb N$ such that $F_{k_n}(x_2)=F_{k_m}(x_2)$, for any $n,m$. Then $F^{-1}_{k_0} \circ F_{k_n} (x_2)=x_2$ and so $F^{-1}_{k_0} \circ F_{k_n}$ fixes $x_0$,  $x_1$ and $x_2$.

\smallskip

Consequently, $F^{-1}_{k_0} \circ F_{k_n}=Id$, since a non trivial quasi-rotation has exactly two fixed points. Finally, $F_{k_0}=F_{k_n}$, for any $n\in \N$. This is a contradiction. Hence $\mathrm{St}_G (x_0)$ is  finite.

\medskip

 Finally, we conclude that $G$ is finite, by proving that $\displaystyle G=\bigcup_{i=0}^n g_i (\mathrm{St}_G(x_0)) $: let $g\in G$, as $g(x_0) \in \mathcal O_G({x_0})$, there exists $i$ such that $g(x_0)=g_i(x_0)$. Hence $g_i^{-1} g\in \mathrm{St}_G(x_0)$ and then
$ g\in g_i(\mathrm{St}_G(x_0))$. \end{proof}

\section{ Burnside problem for groups of homeomorphisms of  the remain surfaces.}

In this section we prove that a finitely generated periodic group of homeomorphism of the closed disk is finite. We also prove Corollary \ref{Ann} and Theorem \ref{theo2}.

\subsection{ Burnside problem for groups of homeomorphisms of the closed 2-disk. } \

\medskip

Let $\Gamma$ be a finitely generated periodic subgroup of homeomorphisms of $\DD^2$ and let $C=\partial \DD^2$. As $C$ is invariant by $\Gamma$, according to the positive answer to the Burnside problem on the circle, $\Gamma$ acts as an abelian and finite group  on $C$ and this action is faithful since  any periodic homeomorphism on $\DD^2$ whose restriction to $C$ is the identity is the identity (see  \cite{CK}). As a consequence,  $\Gamma$ is an abelian and finite group.

\subsection{ Burnside problem for groups of homeomorphisms of the closed annulus (Corollary \ref{Ann}) } \

\medskip
{\bf First Proof using Burnside Problem on $\TT^2$.}

Let $\Gamma$ be a finitely generated periodic subgroup of $\mathrm{Homeo}(\AA^2)$.  We can form $\TT^2$ the double of $\AA^2$ by gluing two copies of $\AA^2$: $A_1$ and $A_2$ along their boundary.

\smallskip

 Let $g\in \Gamma$, we denote by $g_i$ its corresponding map on $A_i$.  We define the double $\tilde g$  of $g$ on $\TT^2$ by $\tilde g = g_i(x) $ if $x\in A_i$. By construction $\tilde g$ is a finite order (same order as $g$)  homeomorphism that preserves the gluing boundary $C$  and the induced action of $\Gamma$ ($g \mapsto \tilde g$) is faithful.

 \smallskip

Therefore $\Gamma$ acts on $C$ as a finitely generated periodic group and its subgroup preserving each boundary component is of index at most $2$. According to the positive answer to the Burnside problem on the circle, this subgroup acts on each component as a finite group and the same holds for $\Gamma$. We conclude that $\Gamma$ acts on $\TT^2$ with a finite orbit. In particular $\Gamma$ acts faithfully and  preserving a probability measure on $\TT^2$,  the main theorem of \cite{GL3} implies that $\Gamma$ is finite.

 \medskip
{\bf Second Proof.}
Let $g \in \Gamma$. Any finite order  homeomorphism on the closed annulus is an isometry for some flat Riemannian metric (see \cite{Ep} and \cite{St}).
Let $\Gamma_0$ be the subgroup of $\Gamma$ consisting in homeomorphisms that preserve any connected component of the boundary. It is of finite index, so for proving that $\Gamma$ is finite it suffices to prove that its subgroup $\Gamma_0$,  is finite.
Let $C_1$ and $C_2$ be the connected components of the boundary. As $C_i$ is invariant by $\Gamma_0$, $\Gamma_0$ acts as  an abelian and finite group  on $C_i$ and this action is faithful, since an orientation preserving isometry on closed annulus is a rotation about the  central axis, if it is equal to identity on $C_1 \cup C_2$, then it is the identity on the annulus.

\medskip
The forthcoming subsections provide the proof of Theorem \ref{theo2}.
\subsection{ Burnside problem for groups of homeomorphisms of a closed orientable surface $S$ of genus $g$ greater than one. } \

\smallskip

Let $\Gamma$ be a finitely generated periodic subgroup of homeomorphisms of  $S$.

Any finite order  homeomorphism on the $S$ is an isometry for some Riemannian metric of constant curvature equal to $-1$ (see \cite{Ep} and \cite{St}).

Let $\phi: \Gamma \rightarrow \mathrm{GL}(2g, \ZZ)$ be the homology morphism.

The image $\phi(\Gamma)$  is a finitely generated periodic subgroup of  $\mathrm{GL}(2g,\ZZ)$, hence it is finite, according to Schur's theorem.

We will use a strong result: there is no torsion element in the Torelli's group, that is, any homeomorphism homologous to identity whose isotopic class is periodic is isotopic to identity (see, for example, Theorem 6.12 of \cite{FM}).

Then the kernel of $\phi$ consists in isometries isotopic to identity. But, the only isometry isotopic to identity that preserves orientation is the identity (see \cite{FM} proof of Proposition 7.7). Since the subgroup of  isometries isotopic to identity that preserve orientation is of finite index  in the group  of isometries isotopic to identity, it follows that  the kernel of $\phi$ is  also finite. As a consequence, $\Gamma$ is finite.

\smallskip
\subsection{ Burnside problem for groups of homeomorphisms of a compact orientable surface $S$ whose boundary has more than three circles. } \

\smallskip

Let $\Gamma$ be a finitely generated periodic subgroup of homeomorphisms of  $S$.
$\Gamma$ preserves the boundary of $S$, that is, an union of circles.

 We can form the double of $S$, $\widehat{S}$, by gluing two copies of $S$: $S_1$ and $S_2$ along their respective boundaries $\gamma_1$ and $\gamma_2$.

Let $g\in \Gamma$, we denote by $g_i$ its corresponding map on $S_i$.  We define the double $\tilde g$  of $g$ on $\widehat{S}$ by $\tilde g = g_i(x) $ if $x\in S_i$. By construction $\tilde g$ is a finite order (same order as $g$)  homeomorphism that preserves the gluing boundary $C$  and the induced action of $\Gamma$ ($g \mapsto \tilde g$) is faithful on the compact surface  $\widehat{S}$ of genus greater than one. By last subsection we have that $\Gamma$ is  finite.


%
%
%
%
%
%
%

\section{Proof of Theorem \ref{theo3} and Corollary \ref{coro21}.}

\noindent {\bf Proof of Theorem \ref{theo3}.}

{\bf 1}. Let $\Gamma$  be a finitely generated periodic group of diffeomorphisms of a  compact manifold $M$ of dimension $n$. Suppose that $\Gamma$ acts on $M$ with a finite orbit.

\begin{clai} \label{clai1} There exists $x_0\in M$ such that $\Gamma_0= St_\Gamma ({x_0})$ is a finite index, finitely generated periodic  subgroup of $\Gamma$.
\end{clai}

\begin{proof} Let $\mathcal{O} _{x_0} =\{x_0,x_1,...,x_m\}$ be a finite $\Gamma$-orbit. For $i\in \{0,...,m\}$,  we write $x_i = g_i (x_0)$, where $g_i\in \Gamma$ and $g_0=Id_M$.

Let $g\in \Gamma$, as $g(x_0) \in \mathcal{O} _{x_0}$ there exists some $i \in \{0,...,m\}$ such that $g(x_0) = g_i(x_0)$. Therefore $g_i^{-1} g (x_0) = x_0$ that is  $g\in g_i( St_\Gamma ({x_0}))$.

In conclusion, $\displaystyle G= \bigcup_{i=0}^m g_i( St_\Gamma ({x_0}))$, meaning that $St_\Gamma ({x_0})$ has finite index in $G$ ; according to Schreier's Lemma it is finitely generated.
\end{proof}

\begin{rema} As a consequence of this claim, for proving that $\Gamma$ is finite it suffices to prove that its subgroup $\Gamma_0$, acting with a global fixed point on $M$, is finite. This is the purpose of the next proposition.
\end{rema}

\begin{prop} Let $\Gamma_0$ be finitely generated periodic group of diffeomorphisms of $M$. If $\Gamma_0$ acts on $M$ with a global fixed point then $\Gamma_0$ is finite
\end{prop}

\begin{proof}

Consider the map  $\mathbf {D}:  \Gamma_0 \rightarrow \mathrm{GL}(n,\R)$, where $D(g) = D g(x_0)$ is  the differential map  of $g$  at $x_0$ (after identification of $T_{x_0}M $ to $\RR^n$).

\medskip

It is easy to see that the map $D$ is a morphism and it is faithful.
Indeed, let $g\in \Gamma_0$ such that $D(g) =Id$. We have already noted that $g$ is an isometry for some Riemannian metric ($m_g$) on $M$. Since an isometry is uniquely determined by its value and its differential at a single point, we get that $g=Id_M$.

Then $\Gamma_0$ is isomorphic to its image $D(\Gamma_0)$ which is a finitely generated periodic subgroup of  $\mathrm{GL}(n,\R)$, hence finite, according to Schur's theorem. This concludes the case where $\Gamma$ acts with a finite orbit on $M$

\bigskip

{\bf 2.} Let $\Gamma$  be a finitely generated periodic group of diffeomorphisms of a compact  manifold $M$.  Suppose that $\Gamma$ preserves a finite union of circles on $M$.

By an analogous argument to Claim \ref{clai1}, it suffices to prove that a finitely generated periodic subgroup  of diffeomorphisms of $M$ that preserves a circle is finite. According to the positive answer to the Burnside problem on the circle, $\Gamma$ admits a finite orbit (on its invariant circle) so by part {\bf 1}  it is finite.
\end{proof}

\noindent {\bf Proof of  Corollary \ref{coro21}.}

Let $f$ be a non trivial central element of $\Gamma$, a finitely generated periodic subgroup of $Diff^1 (\SS^2)$. As $f$ commutes with any $g$ in $\Gamma$, the set $Fix f$ consisting of its fixed points is $G$-invariant. The classification of finite order homeomorphisms of $\SS^2$ indicates that $Fix f $ is either finite or a circle. Hence, Theorem \ref{theo3} implies that $\Gamma$ is finite. \hfill $\square$

\bigskip

\section{Groups  of even bounded orders}

The aim of this section is  proving Corollary \ref{coroeven}.

\smallskip

Let $G$ be a finitely generated periodic group of spherical diffeomorphisms with even bounded orders. The subgroup of orientation preserving elements of $G$ is a group with even bounded orders of index at most 2 in $G$, then we can suppose that $G$ consists in $C^1$ quasi-rotations, in particular any non trivial element of $G$ has
exactly two fixed points.

\medskip

According to the classification of the finite subgroups of $Diff_+^1 (\SS^2)$  and the fact that alternating groups $A_4$, $A_5$ and symmetric group $S_4$ contain elements of order 3, a finite group with even orders of orientation preserving spherical diffeomorphisms is either

\begin{enumerate}
\item a cyclic group $\ZZ / m\ZZ$ where
 {$m=2 p$}, $p\in \NN$ or
\item a diedral group $\DD_m =<\sigma, \tau \vert \sigma^2=\tau^m= 1, \sigma\tau\sigma= \tau ^{-1} > \\ = <\sigma, \sigma' \vert   \sigma^2= (\sigma') ^2= (\sigma \sigma')^m=1>$,
  {$m=2 p$},  $p\in \NN$.
\end{enumerate}

\medskip

Note that a group with even orders always contains involutions (elements of order 2) and let us denote $Inv(G)=\{\sigma \in G\setminus{Id} : \sigma^2=1\}$ and  $Z(\sigma)$ the centralizer of $\sigma$ in $G$, that is $Z(\sigma_0)= \{f\in G : f\sigma_0=\sigma_0f \}$.

\medskip

\begin{lemm}\label{carr} \

Let $G$ be a finitely generated periodic group  of orientation preserving spherical diffeomorphisms with even bounded orders.

\begin{enumerate}

\item Let $\sigma_0\in Inv(G)$, the set $Z(\sigma_0)\cap Inv(G)$ is finite,

\item $Inv(G)$ is finite.
\end{enumerate}

\end{lemm}

\begin{proof}


Let us write $Z(\sigma_0)\cap Inv(G)= \{i_n,  n\in \NN\}$, where $i_0=\sigma_0$. 

Fix $n\in\NN$, the group $G_n$ generated by $i_0$, ..., $i_n$ is finitely generated and fixes the set $Fix(i_0)$ (since $i_k$ commutes with $i_0$), then it is finite, by Theorem \ref{theo1}. Therefore this group is either a cyclic group or a diedral group and $ \{Id\} \subset G_0=<i_0> \subset \cdots \subset G_{n} \subset G_{n+1}\cdots$ .

This sequence stabilizes at some rank $n_0$, since $G$ is of bounded orders. That is, for all $n\geq n_0$ one has  $G_n= G_{n_0}$, hence $i_n \in G_{n_0}$, $\forall n\geq n_0$ and therefore $Z(\sigma_0)\cap Inv(G)$ is finite.


\medskip

We now prove item 2. Suppose by contradiction that there exists in $G$  an infinite sequence of involutions $\sigma_1, \cdots ,\sigma_n, \cdots$ .

For all  $n\in \NN$, the group $<\sigma_0, \sigma_n>$ is either cyclic or diedral. Therefore it contains an involution $i_n$  that commutes with $\sigma_0$ and $\sigma_n$ ($i_n=\sigma_0 $ in the cyclic case or $i_n=( \sigma_0\sigma_n) ^{\frac{m_n}{2}}$  in the diedral case $<\sigma_0, \sigma_n>= \DD_{m_n}$).

Since $Z(\sigma_0)\cap Inv(G)$ is finite, we can suppose (eventually passing to an infinite subsequence of $(\sigma_n)$) that  $i_n=i$, for all $n$.

Finally, $i$ commutes with all $\sigma_n$, that is $\{\sigma_n, n\in \NN\}  \subset Z(i)\cap Inv(G)$  which is finite by item 1, that is a contradiction. \end{proof}

\medskip

\noindent {\it End of proof of Corollary \ref{coroeven}.}

Applying Lemma \ref{carr}, we obtain that $Inv(G)$ and therefore {$Fix(Inv(G))=\{x \in \SS^2 : \sigma(x)=x, $ for some $\sigma \in Inv(G) \}$}  are finite sets.

As $Fix (f\sigma f^{-1})= f(Fix(\sigma))$  and $ f\sigma f^{-1} \in Inv(G)$  for all $f\in G$, the finite set  $Fix(Inv(G))$ is non empty and $G$-invariant. By Theorem \ref{theo1}, $G$ is finite. \hfill $\square$

\bigskip

\bigskip

\end{document}